\documentclass[12pt]{amsart} 

\usepackage[left=2.54cm, right=2.54cm, top=2.54cm, bottom=2.54cm]{geometry}

\usepackage[latin1]{inputenc}
\usepackage{graphicx,color} 
 \usepackage{amsmath}
  \usepackage{amssymb,bbm}
  \usepackage{amscd}
\usepackage{enumerate}
\usepackage{multirow}
\usepackage{hyperref}
\usepackage{tikz}
\usetikzlibrary{arrows}
\usepackage{comment}

\renewcommand{\L}{\mathcal L}

\newcommand{\des}{\operatorname{des}}

\newcommand{\be}{\begin{equation}}
\newcommand{\ee}{\end{equation}}
\newcommand{\ds}{\displaystyle}

\newcommand{\im}{\operatorname{im}}

\newcommand{\LL}{\mathcal{L}}

\newcommand{\R}{\mathcal{R}}
\newcommand{\rfactor}{\operatorname{rfactor}}
\newcommand{\Rfactor}{\operatorname{Rfactor}}

\newcommand{\Monoid}{\mathcal{M}}
\newcommand{\prom}{\hat\partial}

\newtheorem{thm}{Theorem}[section]
\newtheorem{cor}[thm]{Corollary}
\newtheorem{lem}[thm]{Lemma}
\newtheorem{prop}[thm]{Proposition}

\newtheorem{rem}[thm]{Remark}
\newtheorem{conj}[thm]{Conjecture}
\newtheorem{eg}[thm]{Example}

\numberwithin{equation}{section}

\makeatletter
\def\Ddots{\mathinner{\mkern1mu\raise\p@
\vbox{\kern7\p@\hbox{.}}\mkern2mu
\raise4\p@\hbox{.}\mkern2mu\raise7\p@\hbox{.}\mkern1mu}}
\makeatother

\begin{document} 
\title{Markov chains for promotion operators}

\author[A. Ayyer]{Arvind Ayyer}
\address[Arvind Ayyer]{Department of Mathematics, UC Davis, One Shields Ave., Davis, CA 95616-8633, U.S.A. \newline
New address: Department of Mathematics, Department of Mathematics, Indian Institute of Science, Bangalore - 560012, India.
}

\email{arvind@math.iisc.ernet.in}

\author[S. Klee]{Steven Klee}
\address[Steven Klee]{Seattle University, Department of Mathematics, 901 12th Avenue E, Seattle, WA 98122, U.S.A}
\email{klee@math.ucdavis.edu}

\author[A. Schilling]{Anne Schilling}
\address[Anne Schilling]{Department of Mathematics, UC Davis, One Shields Ave., Davis, CA 95616-8633, U.S.A.}
\email{anne@math.ucdavis.edu}

\thanks{A.A. would like to acknowledge support from MSRI, where part of this work was done.
S.K. was supported by NSF VIGRE grant DMS--0636297. 
A.S. was supported by NSF grant DMS--1001256 and OCI--1147247.}



\dedicatory{Dedicated to Mohan Putcha and Lex Renner on the occasion of their 60th birthdays.
}

\begin{abstract}
We consider generalizations of Sch\"utzenberger's promotion operator
on the set $\mathcal{L}$ of linear extensions of a finite poset. This
gives rise to a strongly connected graph on $\mathcal{L}$.  In earlier
work \cite{AKS:2012}, we studied promotion-based Markov chains on
these linear extensions which generalizes results on the Tsetlin
library. We used the theory of $\R$-trivial monoids in an essential
way to obtain explicitly the eigenvalues of the transition matrix in
general when the poset is a rooted forest.  We first survey these
results and then present explicit bounds on the mixing time and
conjecture eigenvalue formulas for more general posets. We also
present a generalization of promotion to arbitrary subsets of the
symmetric group.
\end{abstract}

\maketitle

%%%%%%%%%%%%%%%%%%%%%%%%%%%%%%%%%%%%%%%%%%%%%%%%%%%
\section{Introduction}
%%%%%%%%%%%%%%%%%%%%%%%%%%%%%%%%%%%%%%%%%%%%%%%%%%%

Sch\"utzenberger~\cite{schuetzenberger.1972} introduced the notion of
evacuation and promotion on the set of linear extensions of a finite
poset $P$ of size $n$. This generalizes promotion on standard Young
tableaux defined in terms of jeu-de-taquin
moves. Haiman~\cite{haiman.1992} as well as Malvenuto and
Reutenauer~\cite{malvenuto_reutenauer.1994} simplified
Sch\"utzenberger's approach by expressing the promotion operator
$\partial$ in terms of more fundamental operators $\tau_i$ ($1\le
i<n$), which either act as the identity or as a simple
transposition. A beautiful survey on this subject was written by
Stanley~\cite{stanprom}.

In earlier work, we considered a slight generalization of the
promotion operator \cite{AKS:2012} defined as $\partial_i = \tau_i
\tau_{i+1} \cdots \tau_{n-1}$ for $1\le i\le n$ with
$\partial_1=\partial$ being the original promotion operator.  In
Section~\ref{section.promotion} we define the extended promotion
operator, give examples and state some of its properties.  We survey
our results on Markov chains based on the operators $\partial_i$,
which act on the set of all linear extensions of $P$ (denoted
$\mathcal{L}(P)$) in Section~\ref{section.markov chains}.

Our results \cite{AKS:2012} can be viewed as a natural generalization
of the results of Hendricks~\cite{hendricks1,hendricks2} on the
Tsetlin library \cite{tsetlin}, which is a model for the way an
arrangement of books in a library shelf evolves over time. It is a
Markov chain on permutations, where the entry in the $i$th position is
moved to the front with probability $p_i$. From our viewpoint,
Hendricks' results correspond to the case when $P$ is an anti-chain
and hence $\mathcal{L}(P)=S_n$ is the full symmetric group.  Many
variants of the Tsetlin library have been studied and there is a
wealth of literature on the subject. We refer the interested reader to
the monographs by Letac~\cite{letac} and by Dies~\cite{dies}, as well
as the comprehensive bibliographies in~\cite{fill.1996}
and~\cite{bidigare_hanlon_rockmore.1999}.

One of the most interesting properties of the Tsetlin library Markov
chain is that the eigenvalues of the transition matrix can be computed
exactly. The exact form of the eigenvalues was independently
investigated by several groups. Notably Donnelly~\cite{donnelly.1991},
Kapoor and Reingold~\cite{kapoor_reingold.1991}, and
Phatarfod~\cite{phatarfod.1991} studied the approach to stationarity
in great detail.  There has been some interest in finding exact
formulas for the eigenvalues for generalizations of the Tsetlin
library. The first major achievement in this direction was to
interpret these results in the context of hyperplane arrangements
\cite{bidigare.1997, bidigare_hanlon_rockmore.1999,
  brown_diaconis.1998}.  This was further generalized to a class of
monoids called left regular bands~\cite{brown.2000} and subsequently
to all bands~\cite{brown.2004} by Brown. This theory has been used
effectively by Bj\"orner~\cite{bjorner.2008, bjorner.2009} to extend
eigenvalue formulas on the Tsetlin library from a single shelf to
hierarchies of libraries.

We present without proof our explicit combinatorial formulas
\cite{AKS:2012} for the eigenvalues and multiplicities for the
transition matrix of the promotion Markov chain when the underlying
poset is a rooted forest in Section~\ref{section.chains} (see
Theorem~\ref{theorem.eigenvalues}). The proof of eigenvalues
and their multiplicities follows from the $\R$-triviality of the
underlying monoid  using results by Steinberg~\cite{steinberg.2006, steinberg.2008}.
Intuition on why the promotion monoid is $\R$-trivial is stated in Section~\ref{section.R trivial}.

The remainder of the paper contains new results and is outlined as
follows.  In Section~\ref{section.mixing times}, we prove a formula
for the mixing time of the promotion Markov chain. This improves the
result stated without proof in the Outlook section of \cite{AKS:2012}.
In Section~\ref{section.other posets}, we present a partial conjecture
for the eigenvalues of the transition matrix of posets which are not
rooted forests. We give supporting data for our conjectures with
formulas for all posets of size 4. Lastly,
Section~\ref{section.subsets} defines a generalization of promotion on
arbitrary subsets of $S_n$ and gives a formula for its stationary
distribution.

%%%%%%%%%%%%%%%%%%%%%%%%%%%%%%%%%%%%%%%%%%%%%%%%%%%
\subsection*{Acknowledgements}
We would like to thank the organizers Mahir Can, Zhenheng Li, Benjamin Steinberg, and Qiang Wang 
of the workshop on ``Algebraic monoids, group embeddings and algebraic combinatorics" held
July 3-6, 2012 at the Fields Institute at Toronto for giving us the opportunity
to present this work. We would like to thank Nicolas M. Thi\'ery for discussions.

The Markov chains presented in this paper are implemented in a {\tt
  Maple} package by the first author (AA) available from his home page
and in {\tt Sage}~\cite{sage,sage-combinat} by the third author (AS).  Many of the pictures
presented here were created with {\tt Sage}.

%%%%%%%%%%%%%%%%%%%%%%%%%%%%%%%%%%%%%%%%%%%%%%%%%%%
\section{Extended promotion on linear extensions}
\label{section.promotion}
%%%%%%%%%%%%%%%%%%%%%%%%%%%%%%%%%%%%%%%%%%%%%%%%%%%

Let  $P$ be an arbitrary poset of size $n$, with partial order denoted by $\preceq$. We assume that the 
elements of $P$ are labeled by integers in $[n]:=\{1,2,\ldots,n\}$. In addition, we assume
that the poset is naturally labeled, that is if $i,j \in P$ with $i \preceq j$ in $P$ then $i \le j$ as integers.
Let $\L:=\L(P)$ be the set of its {\bf linear extensions}, 
\be
\L(P) = \{ \pi \in S_{n} \mid i \prec j \text{ in $P$ } \implies \pi^{-1}_{i} < \pi^{-1}_{j} \text{ as integers} \},
\ee
which is naturally interpreted as a subset of the symmetric group $S_n$. Note that the identity permutation $e$ 
always belongs to $\L$. Let $P_{j}$ be the natural (induced) subposet of $P$ consisting of elements $k$ 
such that $j \preceq k$~\cite{stanenum}.

We now briefly recall the idea of {\bf promotion} of a linear extension of a poset $P$. Start with a linear extension 
$\pi \in \L(P)$ and imagine placing the label $\pi^{-1}_{i}$ in $P$ at the location $i$. By the definition of the linear 
extension, the labels will be well-ordered. The action of promotion of $\pi$ will give another linear extension of
$P$ as follows:
\begin{enumerate} 

\item The process starts with a seed, the label 1. First remove it and replace it by the minimum of all the labels covering it, $i$, say. 
 
\item Now look for the minimum of all labels covering $i$ in the original poset, and replace it, and continue in this way. 

\item This process ends when a label is a ``local maximum.'' Place the label $n+1$ at that point.

\item Decrease all the labels by 1.

\end{enumerate} 

This new linear extension is denoted $\pi \partial$  \cite{stanprom}.

\begin{eg}
\label{example.promotion slide}
Figure~\ref{figure:promotion-example} shows a poset (left) to which we assign the identity linear extension
$\pi = 123456789$. The linear extension $\pi'=\pi\partial = 214537869$ obtained by applying the promotion operator 
is depicted on the right. Note that indeed we place $\pi_i^{'-1}$ in position $i$, namely 3 is in position 5 in $\pi'$, so that 
5 in $\pi \partial$ is where 3 was originally.

\begin{center}
\begin{figure}[h]
\begin{tabular}{p{2in}p{2in}}
\begin{tikzpicture}[>=latex,line join=bevel,]
\node (1) at (66bp,7bp) [draw,draw=none] {$1$};
  \node (3) at (36bp,57bp) [draw,draw=none] {$3$};
  \node (2) at (36bp,7bp) [draw,draw=none] {$2$};
  \node (5) at (66bp,107bp) [draw,draw=none] {$5$};
  \node (4) at (66bp,57bp) [draw,draw=none] {$4$};
  \node (7) at (6bp,107bp) [draw,draw=none] {$7$};
  \node (6) at (36bp,107bp) [draw,draw=none] {$6$};
  \node (9) at (21bp,157bp) [draw,draw=none] {$9$};
  \node (8) at (96bp,107bp) [draw,draw=none] {$8$};
  \draw [black,-] (9) ..controls (28.603bp,131.66bp) and (31.916bp,120.61bp)  .. (6);
  \draw [black,-] (5) ..controls (66bp,81.179bp) and (66bp,70.462bp)  .. (4);
  \draw [black,-] (3) ..controls (50.771bp,32.382bp) and (57.698bp,20.837bp)  .. (1);
  \draw [black,-] (7) ..controls (20.771bp,82.382bp) and (27.698bp,70.837bp)  .. (3);
  \draw [black,-] (3) ..controls (36bp,31.179bp) and (36bp,20.462bp)  .. (2);
  \draw [black,-] (4) ..controls (66bp,31.179bp) and (66bp,20.462bp)  .. (1);
  \draw [black,-] (9) ..controls (13.397bp,131.66bp) and (10.084bp,120.61bp)  .. (7);
  \draw [black,-] (8) ..controls (81.229bp,82.382bp) and (74.302bp,70.837bp)  .. (4);
  \draw [black,-] (6) ..controls (36bp,81.179bp) and (36bp,70.462bp)  .. (3);
\end{tikzpicture}%
&
\begin{tikzpicture}[>=latex,line join=bevel,]
\node (1) at (66bp,7bp) [draw,draw=none] {$2$};
  \node (3) at (36bp,57bp) [draw,draw=none] {$5$};
  \node (2) at (36bp,7bp) [draw,draw=none] {$1$};
  \node (5) at (66bp,107bp) [draw,draw=none] {$4$};
  \node (4) at (66bp,57bp) [draw,draw=none] {$3$};
  \node (7) at (6bp,107bp) [draw,draw=none] {$6$};
  \node (6) at (36bp,107bp) [draw,draw=none] {$8$};
  \node (9) at (21bp,157bp) [draw,draw=none] {$9$};
  \node (8) at (96bp,107bp) [draw,draw=none] {$7$};
  \draw [black,-] (9) ..controls (28.603bp,131.66bp) and (31.916bp,120.61bp)  .. (6);
  \draw [black,-] (5) ..controls (66bp,81.179bp) and (66bp,70.462bp)  .. (4);
  \draw [black,-] (3) ..controls (50.771bp,32.382bp) and (57.698bp,20.837bp)  .. (1);
  \draw [black,-] (7) ..controls (20.771bp,82.382bp) and (27.698bp,70.837bp)  .. (3);
  \draw [black,-] (3) ..controls (36bp,31.179bp) and (36bp,20.462bp)  .. (2);
  \draw [black,-] (4) ..controls (66bp,31.179bp) and (66bp,20.462bp)  .. (1);
  \draw [black,-] (9) ..controls (13.397bp,131.66bp) and (10.084bp,120.61bp)  .. (7);
  \draw [black,-] (8) ..controls (81.229bp,82.382bp) and (74.302bp,70.837bp)  .. (4);
  \draw [black,-] (6) ..controls (36bp,81.179bp) and (36bp,70.462bp)  .. (3);
\end{tikzpicture}
\end{tabular}
\caption{A linear extension $\pi$ (left) and $\pi\partial$ (right).}
\label{figure:promotion-example}
\end{figure}
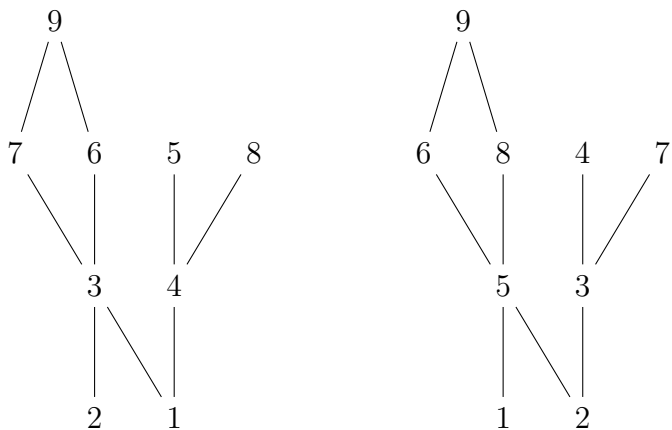
\end{center}

\noindent
Figure \ref{figure:promotion-example2} illustrates the steps used to
construct the linear extension $\pi\partial$ from the linear extension
$\pi$ from Figure \ref{figure:promotion-example}.

\begin{center}
\begin{figure}[h]
\begin{tabular}{|p{1.75in}|p{1.75in}|p{1.75in}|} \hline
\textbf{Step 1:} Remove the minimal element $1$. &
\textbf{Step 2:}    The minimal element that covered $1$ was $3$, so replace $1$ with $3$. &
\textbf{Step 2 (continued):}  The minimal element that covered $3$ was $6$, so replace $3$ with $6$. \\ \hline
%%%%%%%%%%%%%%%%%%%%%%%%%%%%%%
\begin{center}
\begin{tikzpicture}[>=latex,line join=bevel,]
\node (1) at (66bp,7bp) [draw,draw=none] {};
  \node (3) at (36bp,57bp) [draw,draw=none] {$3$};
  \node (2) at (36bp,7bp) [draw,draw=none] {$2$};
  \node (5) at (66bp,107bp) [draw,draw=none] {$5$};
  \node (4) at (66bp,57bp) [draw,draw=none] {$4$};
  \node (7) at (6bp,107bp) [draw,draw=none] {$7$};
  \node (6) at (36bp,107bp) [draw,draw=none] {$6$};
  \node (9) at (21bp,157bp) [draw,draw=none] {$9$};
  \node (8) at (96bp,107bp) [draw,draw=none] {$8$};
  \draw [black,-] (9) ..controls (28.603bp,131.66bp) and (31.916bp,120.61bp)  .. (6);
  \draw [black,-] (5) ..controls (66bp,81.179bp) and (66bp,70.462bp)  .. (4);
  \draw [black,-] (3) ..controls (50.771bp,32.382bp) and (57.698bp,20.837bp)  .. (1);
  \draw [black,-] (7) ..controls (20.771bp,82.382bp) and (27.698bp,70.837bp)  .. (3);
  \draw [black,-] (3) ..controls (36bp,31.179bp) and (36bp,20.462bp)  .. (2);
  \draw [black,-] (4) ..controls (66bp,31.179bp) and (66bp,20.462bp)  .. (1);
  \draw [black,-] (9) ..controls (13.397bp,131.66bp) and (10.084bp,120.61bp)  .. (7);
  \draw [black,-] (8) ..controls (81.229bp,82.382bp) and (74.302bp,70.837bp)  .. (4);
  \draw [black,-] (6) ..controls (36bp,81.179bp) and (36bp,70.462bp)  .. (3);
\end{tikzpicture}%
\end{center}

%%%%%%%%%%%%%%%%%%%%%%%%%%%%%%
&
%%%%%%%%%%%%%%%%%%%%%%%%%%%%%%
\begin{center}
\begin{tikzpicture}[>=latex,line join=bevel,]
\node (1) at (66bp,7bp) [draw,draw=none] {$3$};
  \node (3) at (36bp,57bp) [draw,draw=none] {};
  \node (2) at (36bp,7bp) [draw,draw=none] {$2$};
  \node (5) at (66bp,107bp) [draw,draw=none] {$5$};
  \node (4) at (66bp,57bp) [draw,draw=none] {$4$};
  \node (7) at (6bp,107bp) [draw,draw=none] {$7$};
  \node (6) at (36bp,107bp) [draw,draw=none] {$6$};
  \node (9) at (21bp,157bp) [draw,draw=none] {$9$};
  \node (8) at (96bp,107bp) [draw,draw=none] {$8$};
  \draw [black,-] (9) ..controls (28.603bp,131.66bp) and (31.916bp,120.61bp)  .. (6);
  \draw [black,-] (5) ..controls (66bp,81.179bp) and (66bp,70.462bp)  .. (4);
  \draw [black,-] (3) ..controls (50.771bp,32.382bp) and (57.698bp,20.837bp)  .. (1);
  \draw [black,-] (7) ..controls (20.771bp,82.382bp) and (27.698bp,70.837bp)  .. (3);
  \draw [black,-] (3) ..controls (36bp,31.179bp) and (36bp,20.462bp)  .. (2);
  \draw [black,-] (4) ..controls (66bp,31.179bp) and (66bp,20.462bp)  .. (1);
  \draw [black,-] (9) ..controls (13.397bp,131.66bp) and (10.084bp,120.61bp)  .. (7);
  \draw [black,-] (8) ..controls (81.229bp,82.382bp) and (74.302bp,70.837bp)  .. (4);
  \draw [black,-] (6) ..controls (36bp,81.179bp) and (36bp,70.462bp)  .. (3);
\end{tikzpicture}%
\end{center}
%%%%%%%%%%%%%%%%%%%%%%%%%%%%%%
&
%%%%%%%%%%%%%%%%%%%%%%%%%%%%%%
\begin{center}
\begin{tikzpicture}[>=latex,line join=bevel,]
\node (1) at (66bp,7bp) [draw,draw=none] {$3$};
  \node (3) at (36bp,57bp) [draw,draw=none] {$6$};
  \node (2) at (36bp,7bp) [draw,draw=none] {$2$};
  \node (5) at (66bp,107bp) [draw,draw=none] {$5$};
  \node (4) at (66bp,57bp) [draw,draw=none] {$4$};
  \node (7) at (6bp,107bp) [draw,draw=none] {$7$};
  \node (6) at (36bp,107bp) [draw,draw=none] {};
  \node (9) at (21bp,157bp) [draw,draw=none] {$9$};
  \node (8) at (96bp,107bp) [draw,draw=none] {$8$};
  \draw [black,-] (9) ..controls (28.603bp,131.66bp) and (31.916bp,120.61bp)  .. (6);
  \draw [black,-] (5) ..controls (66bp,81.179bp) and (66bp,70.462bp)  .. (4);
  \draw [black,-] (3) ..controls (50.771bp,32.382bp) and (57.698bp,20.837bp)  .. (1);
  \draw [black,-] (7) ..controls (20.771bp,82.382bp) and (27.698bp,70.837bp)  .. (3);
  \draw [black,-] (3) ..controls (36bp,31.179bp) and (36bp,20.462bp)  .. (2);
  \draw [black,-] (4) ..controls (66bp,31.179bp) and (66bp,20.462bp)  .. (1);
  \draw [black,-] (9) ..controls (13.397bp,131.66bp) and (10.084bp,120.61bp)  .. (7);
  \draw [black,-] (8) ..controls (81.229bp,82.382bp) and (74.302bp,70.837bp)  .. (4);
  \draw [black,-] (6) ..controls (36bp,81.179bp) and (36bp,70.462bp)  .. (3);
\end{tikzpicture}%
\end{center}
%%%%%%%%%%%%%%%%%%%%%%%%%%%%%%
\\ \hline
%%%%%%%%%%%%%%%%%%%%%%%%%%%%%%
\textbf{Step 2 (continued):} The minimal element that covered $6$ was $9$, so replace $6$ with $9$. &
\textbf{Step 3:} Since $9$ was a local maximum, replace $9$ with $10$. & 
\textbf{Step 4:} Decrease all labels by $1$.  The resulting linear extension is $\partial \pi$.
%%%%%%%%%%%%%%%%%%%%%%%%%%%%%%
\\ \hline
%%%%%%%%%%%%%%%%%%%%%%%%%%%%%%
\begin{center}
\begin{tikzpicture}[>=latex,line join=bevel,]
\node (1) at (66bp,7bp) [draw,draw=none] {$3$};
  \node (3) at (36bp,57bp) [draw,draw=none] {$6$};
  \node (2) at (36bp,7bp) [draw,draw=none] {$2$};
  \node (5) at (66bp,107bp) [draw,draw=none] {$5$};
  \node (4) at (66bp,57bp) [draw,draw=none] {$4$};
  \node (7) at (6bp,107bp) [draw,draw=none] {$7$};
  \node (6) at (36bp,107bp) [draw,draw=none] {$9$};
  \node (9) at (21bp,157bp) [draw,draw=none] {};
  \node (8) at (96bp,107bp) [draw,draw=none] {$8$};
  \draw [black,-] (9) ..controls (28.603bp,131.66bp) and (31.916bp,120.61bp)  .. (6);
  \draw [black,-] (5) ..controls (66bp,81.179bp) and (66bp,70.462bp)  .. (4);
  \draw [black,-] (3) ..controls (50.771bp,32.382bp) and (57.698bp,20.837bp)  .. (1);
  \draw [black,-] (7) ..controls (20.771bp,82.382bp) and (27.698bp,70.837bp)  .. (3);
  \draw [black,-] (3) ..controls (36bp,31.179bp) and (36bp,20.462bp)  .. (2);
  \draw [black,-] (4) ..controls (66bp,31.179bp) and (66bp,20.462bp)  .. (1);
  \draw [black,-] (9) ..controls (13.397bp,131.66bp) and (10.084bp,120.61bp)  .. (7);
  \draw [black,-] (8) ..controls (81.229bp,82.382bp) and (74.302bp,70.837bp)  .. (4);
  \draw [black,-] (6) ..controls (36bp,81.179bp) and (36bp,70.462bp)  .. (3);
\end{tikzpicture}%
\end{center}
%%%%%%%%%%%%%%%%%%%%%%%%%%%%%%
&
%%%%%%%%%%%%%%%%%%%%%%%%%%%%%%
\begin{center}
\begin{tikzpicture}[>=latex,line join=bevel,]
\node (1) at (66bp,7bp) [draw,draw=none] {$3$};
  \node (3) at (36bp,57bp) [draw,draw=none] {$6$};
  \node (2) at (36bp,7bp) [draw,draw=none] {$2$};
  \node (5) at (66bp,107bp) [draw,draw=none] {$5$};
  \node (4) at (66bp,57bp) [draw,draw=none] {$4$};
  \node (7) at (6bp,107bp) [draw,draw=none] {$7$};
  \node (6) at (36bp,107bp) [draw,draw=none] {$9$};
  \node (9) at (21bp,157bp) [draw,draw=none] {$10$};
  \node (8) at (96bp,107bp) [draw,draw=none] {$8$};
  \draw [black,-] (9) ..controls (28.603bp,131.66bp) and (31.916bp,120.61bp)  .. (6);
  \draw [black,-] (5) ..controls (66bp,81.179bp) and (66bp,70.462bp)  .. (4);
  \draw [black,-] (3) ..controls (50.771bp,32.382bp) and (57.698bp,20.837bp)  .. (1);
  \draw [black,-] (7) ..controls (20.771bp,82.382bp) and (27.698bp,70.837bp)  .. (3);
  \draw [black,-] (3) ..controls (36bp,31.179bp) and (36bp,20.462bp)  .. (2);
  \draw [black,-] (4) ..controls (66bp,31.179bp) and (66bp,20.462bp)  .. (1);
  \draw [black,-] (9) ..controls (13.397bp,131.66bp) and (10.084bp,120.61bp)  .. (7);
  \draw [black,-] (8) ..controls (81.229bp,82.382bp) and (74.302bp,70.837bp)  .. (4);
  \draw [black,-] (6) ..controls (36bp,81.179bp) and (36bp,70.462bp)  .. (3);
\end{tikzpicture}%
\end{center}
%%%%%%%%%%%%%%%%%%%%%%%%%%%%%%
&
%%%%%%%%%%%%%%%%%%%%%%%%%%%%%%
\begin{center}
\begin{tikzpicture}[>=latex,line join=bevel,]
\node (1) at (66bp,7bp) [draw,draw=none] {$2$};
  \node (3) at (36bp,57bp) [draw,draw=none] {$5$};
  \node (2) at (36bp,7bp) [draw,draw=none] {$1$};
  \node (5) at (66bp,107bp) [draw,draw=none] {$4$};
  \node (4) at (66bp,57bp) [draw,draw=none] {$3$};
  \node (7) at (6bp,107bp) [draw,draw=none] {$6$};
  \node (6) at (36bp,107bp) [draw,draw=none] {$8$};
  \node (9) at (21bp,157bp) [draw,draw=none] {$9$};
  \node (8) at (96bp,107bp) [draw,draw=none] {$7$};
  \draw [black,-] (9) ..controls (28.603bp,131.66bp) and (31.916bp,120.61bp)  .. (6);
  \draw [black,-] (5) ..controls (66bp,81.179bp) and (66bp,70.462bp)  .. (4);
  \draw [black,-] (3) ..controls (50.771bp,32.382bp) and (57.698bp,20.837bp)  .. (1);
  \draw [black,-] (7) ..controls (20.771bp,82.382bp) and (27.698bp,70.837bp)  .. (3);
  \draw [black,-] (3) ..controls (36bp,31.179bp) and (36bp,20.462bp)  .. (2);
  \draw [black,-] (4) ..controls (66bp,31.179bp) and (66bp,20.462bp)  .. (1);
  \draw [black,-] (9) ..controls (13.397bp,131.66bp) and (10.084bp,120.61bp)  .. (7);
  \draw [black,-] (8) ..controls (81.229bp,82.382bp) and (74.302bp,70.837bp)  .. (4);
  \draw [black,-] (6) ..controls (36bp,81.179bp) and (36bp,70.462bp)  .. (3);
\end{tikzpicture}
\end{center}
%%%%%%%%%%%%%%%%%%%%%%%%%%%%%%

%%%%%%%%%%%%%%%%%%%%%%%%%%%%%%
\\ \hline
\end{tabular}
\caption{Constructing $\pi\partial$ from $\pi$.}
\label{figure:promotion-example2}
\end{figure}
\end{center}
\end{eg}

The definition of promotion was originally motivated by the following construction.  The Young diagram of a partition $\lambda$
(with English notation) can naturally be viewed as a poset on the boxes of the diagram ordered according to the rule that a box
is covered by any boxes immediately below it or to its right.  The linear extensions of this poset are standard Young tableaux of
shape $\lambda$.  In this context, the definition of promotion is a natural generalization of the standard promotion operator
used in the RSK algorithm. On semistandard tableaux, promotion is also used to define affine crystal structures in 
type $A$~\cite{shimozono.2002} and it has applications to the cyclic sieving phenomenon~\cite{rhoades.2010}.
The above definition of promotion for arbitrary posets is originally due to  Sch\"utzenberger~\cite{schuetzenberger.1972}.  

We now generalize the above construction to {\bf extended promotion}, whose seed is any of the numbers $1,2,\ldots,n$. The algorithm 
is similar to the original one, and we describe it for seed $j$. Start with the subposet $P_{j}$ and perform 
steps 1--3 in a completely analogous fashion. Now decrease all the labels strictly larger than $j$ by 1 in $P$ 
(not only $P_{j}$). Clearly this gives a new linear extension, which we denote $\pi \partial_{j}$.
Note that $\partial_n$ is always the identity.

The extended promotion operator can be expressed in terms of more elementary operators $\tau_i$ ($1\le i<n$) as shown 
in~\cite{haiman.1992, malvenuto_reutenauer.1994,stanprom} and has explicitly been used to count linear 
extensions in~\cite{edelman_hibi_stanley.1989}. 
Let $\pi=\pi_1 \cdots \pi_n \in \L(P)$ be a linear extension of a finite poset $P$ in one-line notation. Then
\begin{equation} \label{deftau}
	\pi \tau_i = \begin{cases}
	\pi_1 \cdots \pi_{i-1} \pi_{i+1} \pi_i \cdots \pi_n & \text{if $\pi_i$ and $\pi_{i+1}$ are not}\\
	& \text{comparable in $P$,}\\
	\pi_1 \cdots \pi_n & \text{otherwise.} \end{cases}
\end{equation}
Alternatively, $\tau_i$ acts non-trivially on a linear extension if interchanging entries $\pi_i$ and $\pi_{i+1}$ yields another 
linear extension. Then as an operator on $\L(P)$,
\begin{equation}
\label{equation.promotion tau}
 	\partial_j = \tau_j \tau_{j+1} \cdots \tau_{n-1}.
\end{equation}

%%%%%%%%%%%%%%%%%%%%%%%%%%%%%%%%%%%%%%%%%%%%%%%%%%%
\section{Promotion Markov chains} 
\label{section.markov chains}
%%%%%%%%%%%%%%%%%%%%%%%%%%%%%%%%%%%%%%%%%%%%%%%%%%%

We now consider two discrete-time Markov chains related to the
extended promotion operator.  For completeness, we briefly review the
part of the theory relevant to us.

Fix a finite poset $P$ of size $n$. The operators $\{\partial_i \mid 1\le i\le n\}$ define a directed
graph on the set of linear extensions $\L(P)$. The vertices of the
graph are the elements in $\L(P)$ and there is an edge from $\pi$ to
$\pi'$ if $\pi' = \pi\partial_i$.  We can
now consider random walks on this graph with probabilities given
formally by $x_{1},\dots,x_n$ which sum to 1. We give two
ways to assign the edge weights, see
Sections~\ref{subsection.promotion uniform} and
\ref{subsection.promotion}.  An edge with weight $x_{i}$ is traversed
with that rate.  A priori, the $x_{i}$'s must be positive real numbers
for this to make sense according to the standard techniques of Markov
chains.  However, the ideas work in much greater generality and one
can think of this as an ``analytic continuation.''

A discrete-time Markov chain is defined by the {\bf transition matrix}
$M$, whose entries are indexed by elements of the state space. In our
case, they are labeled by elements of $\L(P)$. We take the convention
that the $(\pi',\pi)$ entry gives the probability of going from $\pi
\to \pi'$. The special case of the diagonal entry at $(\pi,\pi)$ gives
the probability of a loop at the $\pi$.  This ensures that column sums
of $M$ are one and consequently, one is an eigenvalue with row (left-)
eigenvector being the all-ones vector.  A Markov chain is said to be
{\bf irreducible} if the associated digraph is strongly connected. In
addition, it is said to be {\bf aperiodic} if the greatest common
divisor of the lengths of all possible loops from any state to itself
is one.  For irreducible aperiodic chains, the Perron-Frobenius
theorem guarantees that there is a unique {\bf stationary
  distribution}. This is given by the entries of the column (right-)
eigenvector of $M$ with eigenvalue 1. Equivalently, the stationary distribution
$w(\pi)$ is the solution of the {\bf master equation}, given by
\be \label{master.equation} 
\sum_{\pi' \in \L(P)} M_{\pi,\pi'} \;w(\pi') 
= \sum_{\pi' \in \L(P)} M_{\pi',\pi} \; w(\pi).  
\ee 
Edges which are loops contribute to both sides equally and thus cancel
out.  For more on the theory of finite state Markov chains, see
\cite{levin_peres_wilmer.2009}.

We set up a running example that will be used for each case.

\begin{eg} \label{example.running example}
Define $P$ by its covering relations $\{ (1,3), (1,4), (2,3) \}$, so that its Hasse diagram is the first
diagram in the list below:
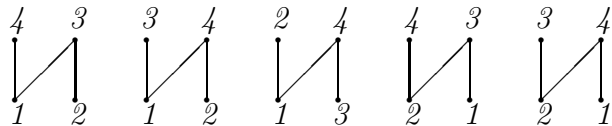
\setlength{\unitlength}{.8mm}
\begin{center}
\begin{picture}(20, 20)
\put(10,4){\circle*{1}}
\put(20,4){\circle*{1}}
\put(9,0){1}
\put(19,0){2}
\put(10,14){\circle*{1}}
\put(20,14){\circle*{1}}
\put(9,16){4}
\put(19,16){3}
\put(10,4){\line(0,1){10}}
\put(20,4){\line(0,1){10}}
\put(10,4){\line(1,1){10}}
\end{picture}
\begin{picture}(20, 20)
\put(10,4){\circle*{1}}
\put(20,4){\circle*{1}}
\put(9,0){1}
\put(19,0){2}
\put(10,14){\circle*{1}}
\put(20,14){\circle*{1}}
\put(9,16){3}
\put(19,16){4}
\put(10,4){\line(0,1){10}}
\put(20,4){\line(0,1){10}}
\put(10,4){\line(1,1){10}}
\end{picture}
\begin{picture}(20, 20)
\put(10,4){\circle*{1}}
\put(20,4){\circle*{1}}
\put(9,0){1}
\put(19,0){3}
\put(10,14){\circle*{1}}
\put(20,14){\circle*{1}}
\put(9,16){2}
\put(19,16){4}
\put(10,4){\line(0,1){10}}
\put(20,4){\line(0,1){10}}
\put(10,4){\line(1,1){10}}
\end{picture}
\begin{picture}(20, 20)
\put(10,4){\circle*{1}}
\put(20,4){\circle*{1}}
\put(9,0){2}
\put(19,0){1}
\put(10,14){\circle*{1}}
\put(20,14){\circle*{1}}
\put(9,16){4}
\put(19,16){3}
\put(10,4){\line(0,1){10}}
\put(20,4){\line(0,1){10}}
\put(10,4){\line(1,1){10}}
\end{picture}
\begin{picture}(20, 20)
\put(10,4){\circle*{1}}
\put(20,4){\circle*{1}}
\put(9,0){2}
\put(19,0){1}
\put(10,14){\circle*{1}}
\put(20,14){\circle*{1}}
\put(9,16){3}
\put(19,16){4}
\put(10,4){\line(0,1){10}}
\put(20,4){\line(0,1){10}}
\put(10,4){\line(1,1){10}}
\end{picture}
\end{center}
The remaining graphs correspond to the linear extensions
$
\L(P) = \{ 1234, 1243, 1423, 2134, 2143 \}.
$
\end{eg}

%%%%%%%%%%%%%%%%%%%%%%%%%%%%%%%%%%%%%%%%%%%%%%%%%%%
\subsection{Uniform promotion graph}
\label{subsection.promotion uniform}

The vertices of the {\bf uniform promotion graph} are labeled by elements of $\L(P)$ and there is an 
edge between $\pi$ and $\pi'$ if and only if $\pi' = \pi \partial_{j}$ for some $j \in [n]$. In this case, the edge 
is given the symbolic weight $x_{j}$.

\begin{eg}
The uniform promotion graph for the poset in Example~\ref{example.running example} 
is illustrated in Figure~\ref{figure.uniform promotion}.
\begin{figure}[h]
\begin{center}
\begin{tikzpicture} [>=triangle 45]
\draw (0,4) node {\verb!1234!};
\draw (0,0) node {\verb!1243!};
\draw (6,0) node {\verb!2134!};
\draw (6,4) node {\verb!2143!};
\draw (3,2) node {\verb!1423!};
\draw [<->] (.1,.3) to[out=80, in=-80] (.1,3.7);
\draw (.6,2) node {$x_3$};
\draw [->] (.5,4) -- (5.5,4);
\draw (3,4.3) node {$x_1$};
\draw [->] (-.1,3.7) to[out=260, in=100] (-.1,.3);
\draw (-.6,2) node {$x_2$};
\draw [<->] (.5,0) -- (5.5,0);
\draw (3,.3) node {$x_1$};
\draw [->] (.3,.3) -- (2.5,1.75);
\draw (1.6,.8) node {$x_2$};
\draw [<->] (6.1,.3) to [out=80, in = -80] (6.1,3.7);
\draw (6.6,2) node {$x_3$};
\draw [<->] (5.9,.3) to [out = 100, in = -100] (5.9,3.7);
\draw (5.4,2) node {$x_2$};
\draw [->] (5.7,3.7) -- (3.5,2.25);
\draw (4.6,3.2) node {$x_1$};
\draw [->] (2.6,2.2) to[out=105, in = -15] (.5,3.8);
\draw (2.5,3) node {$x_1$};
\draw [->] (2.4,2.1) to[out=165, in = -75] (.4,3.6);
\draw (1.5,2.1) node {$x_2$};
\end{tikzpicture}
\end{center}
\caption{Uniform promotion graph for Example~\ref{example.running example}.
Every vertex has four outgoing edges labeled $x_1$ to $x_4$ and self-loops are not drawn.
\label{figure.uniform promotion}}
\end{figure}
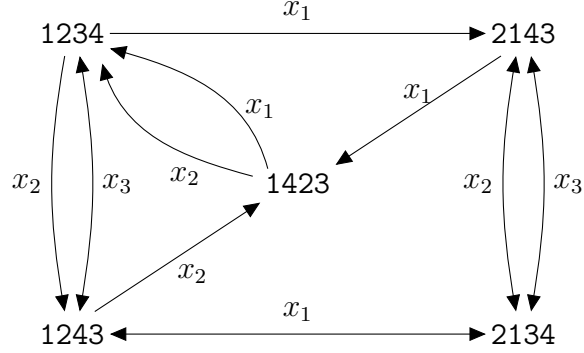
The transition matrix, with the lexicographically ordered basis, is given by
\[
\begin{pmatrix}
x_{4} & x_{3}  & x_{1}+x_{2} & 0 & 0 \\
x_{2}+x_{3} & x_{4} & 0 & x_{1} & 0 \\
0 & x_{2} & x_{3}+x_{4} & 0 & x_{1} \\
0 & x_{1} & 0 & x_{4} & x_{2}+x_{3} \\
x_{1} & 0 & 0 & x_{2} + x_{3} & x_{4}
\end{pmatrix}\;.
\]
Note that the row sums are one although the matrix is not symmetric,
so that the stationary state of this Markov chain is uniform. We state
this for general finite posets in Theorem~\ref{theorem.uniform
  promotion}.

The variable $x_{4}$ occurs only on the
diagonal in the above transition matrix.  This is because the action
of $\partial_{4}$ (or in general $\partial_n$) maps every linear
extension to itself resulting in a loop.

\end{eg}

%%%%%%%%%%%%%%%%%%%%%%%%%%%%%%%%%%%%%%%%%%%%%%%%%%%
\subsection{Promotion graph}
\label{subsection.promotion}

The {\bf promotion graph} is defined in the same fashion as the uniform promotion graph with the exception that
the edge between $\pi$ and $\pi'$ when $\pi' = \pi \partial_{j}$ is given the weight $x_{\pi_j}$. 

\begin{eg} \label{example.promotion}
The promotion graph for the poset of Example~\ref{example.running example}
is illustrated in Figure~\ref{figure.promotion}. Although it might appear that there are many more
edges here than in Figure~\ref{figure.uniform promotion}, this is not the case.
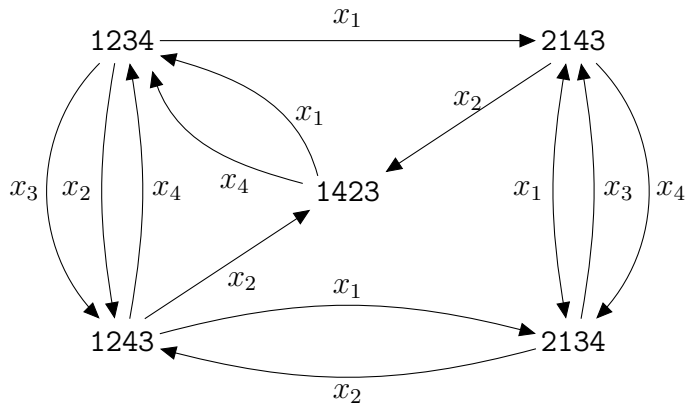
\begin{figure}[h]
\begin{center}
\begin{tikzpicture} [>=triangle 45]
\draw (0,4) node {\verb!1234!};
\draw (0,0) node {\verb!1243!};
\draw (6,0) node {\verb!2134!};
\draw (6,4) node {\verb!2143!};
\draw (3,2) node {\verb!1423!};
\draw [->] (.1,.3) to[out=80, in=-80] (.1,3.7);
\draw (.6,2) node {$x_4$};
\draw [->] (-.1,3.7) to[out=260, in=100] (-.1,.3);
\draw (-.6,2) node {$x_2$};
\draw [->] (-.3,3.7) to[out=225, in = 135] (-.3,.3);
\draw (-1.3,2) node {$x_3$};
\draw [->] (.5,4) -- (5.5,4);
\draw (3,4.3) node {$x_1$};
\draw [->] (.5,.1) to [out = 15, in = 165] (5.5,.1);
\draw (3,.7) node {$x_1$};
\draw [->] (5.5,-.1) to [out = -165, in = -15] (.5,-.1);
\draw (3,-.7) node {$x_2$};
\draw [->] (.3,.3) -- (2.5,1.75);
\draw (1.6,.8) node {$x_2$};
\draw [->] (6.1,.3) to [out=80, in = -80] (6.1,3.7);
\draw (6.6,2) node {$x_3$};
\draw [->] (6.3,3.7) to [out = -45, in = 45] (6.3,.3);
\draw (7.3,2) node {$x_4$};
\draw [<->] (5.9,.3) to [out = 100, in = -100] (5.9,3.7);
\draw (5.4,2) node {$x_1$};
\draw [->] (5.7,3.7) -- (3.5,2.25);
\draw (4.6,3.2) node {$x_2$};
\draw [->] (2.6,2.2) to[out=105, in = -15] (.5,3.8);
\draw (2.5,3) node {$x_1$};
\draw [->] (2.4,2.1) to[out=165, in = -75] (.4,3.6);
\draw (1.5,2.1) node {$x_4$};
\end{tikzpicture}
\caption{Promotion graph for Example~\ref{example.running example}.
Every vertex has four outgoing edges labeled $x_1$ to $x_4$ and self-loops are not drawn.
\label{figure.promotion}}
\end{center}
\end{figure}
The transition matrix this time is given by
\[
\begin{pmatrix}
x_{4} & x_{4}  & x_{1}+x_{4} & 0 & 0 \\
x_{2}+x_{3} & x_{3} & 0 & x_{2} & 0 \\
0 & x_{2} & x_{2}+x_{3} & 0 & x_{2} \\
0 & x_{1} & 0 & x_{4} & x_{1}+x_{4} \\
x_{1} & 0 & 0 & x_{1} + x_{3} & x_{3}
\end{pmatrix}\;.
\]
Notice that row sums are no longer one. The stationary distribution,
as a vector written in row notation is
\[
\begin{pmatrix}
1, &
\ds \frac{x_{1}+x_{2}+x_{3}}{x_{1}+x_{2}+x_{4}}, &
\ds\frac{(x_{1}+x_{2})(x_{1}+x_{2}+x_{3})}
{(x_{1}+x_{2})(x_{1}+x_{2}+x_{4})}, &
\ds \frac{x_{1}}{x_{2}}, &
\ds \frac{x_{1}(x_{1}+x_{2}+x_{3})}
{x_{2}(x_{1}+x_{2}+x_{4})} 
\end{pmatrix}^{T}\; .
\]
Again, we will give a general such result in Theorem~\ref{theorem.promotion}.
\end{eg}

%%%%%%%%%%%%%%%%%%%%%%%%%%%%%%%%%%%%%%%%%%%%%%%%%%%
\subsection{Irreducibility and stationary states}
\label{subsection.stationary}

In this section we summarize some properties of the promotion Markov chains 
of Sections~\ref{subsection.promotion uniform} and~\ref{subsection.promotion} and state
their stationary distributions. Proofs of these statements can be found in~\cite{AKS:2012}.

\begin{prop}
\label{proposition.strongly_connected}
Consider the digraph $G$ whose vertices are labeled by elements of $\L$ and whose edges are given as follows: 
for $\pi, \pi' \in \L$, there is an edge between $\pi$ and $\pi'$ in $G$ if and only if 
$\pi' = \pi \partial_{j}$ for some $j \in [n]$. Then $G$ is strongly connected.
\end{prop}

\begin{cor}
Assuming that the edge weights are strictly positive, the two Markov
chains of Sections~\ref{subsection.promotion uniform}
and~\ref{subsection.promotion} are irreducible and ergodic. Hence
their stationary states are unique.
\end{cor}

Next we state properties of the stationary state of the two discrete
time Markov chains, assuming that all $x_i$'s are strictly positive.

\begin{thm} \label{theorem.uniform promotion}
The discrete time Markov chain according to the uniform promotion
graph has the uniform stationary distribution, that is, each linear
extension is equally likely to occur.
\end{thm}

We now turn to the promotion  graphs.
In this case we find nice product formulas for the stationary weights.

\begin{thm} \label{theorem.promotion}
The stationary state weight $w(\pi)$ of the linear extension $\pi\in \L(P)$ for the discrete time Markov chain for the 
promotion graph is given by
\be \label{formulaII}
w(\pi) = \prod_{i=1}^{n} \frac{x_{1}+ \cdots + x_{i}}
{x_{\pi_1}+ \cdots + x_{\pi_i}}\;,
\ee
assuming $w(e)=1$.
\end{thm}

\begin{rem} \label{remark.constant}
The entries of $w$ do not, in general, sum to one. Therefore this
is not a true probability distribution, but this is easily remedied by a multiplicative constant $Z_{P}$ depending 
only on the poset $P$.
\end{rem}

When $P$ is the $n$-antichain, then $\L = S_{n}$. In this case, the
probability distribution of Theorem~\ref{theorem.promotion} has been
studied in the past by Hendricks \cite{hendricks1,hendricks2} and is
known as the \textbf{Tsetlin library}~\cite{tsetlin}, which we now
describe. Suppose that a library consists of $n$ books
$b_{1},\dots,b_{n}$ on a single shelf. Assume that only one book is
picked at a time and is returned before the next book is picked
up. The book $b_{i}$ is picked with probability $x_{i}$ and placed at
the end of the shelf.

We now explain why promotion on the $n$-antichain is the Tsetlin
library. A given ordering of the books can be identified with a permutation
$\pi$. The action of $\partial_k$ on $\pi$ gives $\pi \tau_k \cdots
\tau_{n-1}$ by \eqref{equation.promotion tau}, where now all the
$\tau_i$'s satisfy the braid relation since none of the $\pi_j$'s are
comparable. Thus the $k$-th element in $\pi$ is moved all the way to
the end. The probability with which this happens is $x_{\pi_k}$, which
makes this process identical to the action of the Tsetlin library.

The stationary distribution of the Tsetlin library
is a special case of Theorem~\ref{theorem.promotion}. In this case, $Z_{P}$ of Remark~\ref{remark.constant} also 
has a nice product formula, leading to the probability distribution,
\be 
w(\pi) = \prod_{i=1}^{n} \frac{x_{\pi_{i}}}{x_{\pi_1}+ \cdots + x_{\pi_i}}.
\ee
Letac~\cite{letac} considered generalizations of the Tsetlin library to rooted trees (meaning that each element in $P$
besides the root has precisely one successor). Our results hold for any finite poset $P$.

%%%%%%%%%%%%%%%%%%%%%%%%%%%%%%%%%%%%%%%%%%%%%%%%%%%
\section{Partition functions and eigenvalues for rooted
forests} \label{section.chains}
%%%%%%%%%%%%%%%%%%%%%%%%%%%%%%%%%%%%%%%%%%%%%%%%%%%

For a certain class of posets, we are able to give an explicit formula
for the probability distribution for the promotion graph.  Note that
this involves computing the partition function $Z_P$ (see
Remark~\ref{remark.constant}). We can also specify all eigenvalues and
their multiplicities of the transition matrix explicitly. Proofs of
these statements can be found in~\cite{AKS:2012}.

Before we can state the main theorems of this section, we need to make
a couple of definitions.  A \textbf{rooted tree} is a connected poset,
where each node has at most one successor.  Note that a rooted tree
has a unique largest element.  A \textbf{rooted forest} is a union of
rooted trees.  A \textbf{lower set} (resp. \textbf{upper set}) $S$ in
a poset is a subset of the nodes such that if $x\in S$ and $y\preceq
x$ (resp. $y\succeq x$), then also $y\in S$.  We first give the
formula for the partition function.

\begin{thm} \label{theorem.partition.function}
Let $P$ be a rooted forest of size $n$ 
and let $x_{\preceq i} = \sum_{j \preceq i} x_j$. 
The partition function for the promotion graph is
given by
\be \label{equation.Zp}
Z_P = \prod_{i=1}^n \frac{x_{\preceq i}}{x_1 + \cdots + x_i}.
\ee
\end{thm}

Let $L$ be a finite poset with smallest element $\hat{0}$ and largest element $\hat{1}$.
Following~\cite[Appendix C]{brown.2000}, one may associate to each element $x\in L$ a \textbf{derangement number}
$d_x$ defined as
\begin{equation}
\label{equation.derangements}
	d_x = \sum_{y\succeq x} \mu(x,y) f([y,\hat{1}])\;,
\end{equation}
where $\mu(x,y)$ is the M\"obius function for the interval 
$[x,y] := \{z \in L \mid x\preceq z \preceq y\}$~\cite[Section 3.7]{stanenum}
and $f([y,\hat{1}])$ is the number of maximal chains in the interval $[y,\hat{1}]$.

A permutation is a \textbf{derangement} if it does not have any fixed
points.  A linear extension $\pi$ is called a \textbf{poset
  derangement} if it is a derangement when considered as a
permutation.  Let $\mathfrak d_P$ be the number of poset derangements
of the poset $P$. 

A \textbf{lattice} $L$ is a poset in which any two elements have a unique supremum (also called join) and 
a unique infimum (also called meet). For $x,y\in L$ the join is denoted by $x \vee y$, whereas the meet is
$x\wedge y$. For an \textbf{upper semi-lattice} we only require the existence of a unique supremum of
any two elements.

\begin{thm} \label{theorem.eigenvalues}
Let $P$ be a rooted forest of size $n$ and $M$ the transition matrix of
the promotion graph of Section~\ref{subsection.promotion}. Then
\begin{equation*}
	\det(M-\lambda \mathbbm{1}) = \prod_{\substack{ S \subseteq
            [n]\\ \text{$S$ upper set in $P$}}} (\lambda - x_S)^{d_S},
\end{equation*}
where $x_S = \sum_{i\in S} x_i$ and $d_S$ is the derangement number in the lattice $L$ (by inclusion) of upper
sets in $P$. In other words, for each subset $S\subseteq [n]$, which is an upper set in $P$, there is an eigenvalue
$x_S$ with multiplicity $d_S$.
\end{thm}

The proof of Theorem~\ref{theorem.eigenvalues} follows from the fact
that the monoid corresponding to the transition matrix $M$
is $\R$-trivial.  
When $P$ is a union of chains, which is a special
case of rooted forests, we can express the eigenvalue multiplicities
directly in terms of the number of poset derangements.

\begin{thm} \label{theorem.poset derangements}
Let $P = [n_1] + [n_2] + \cdots + [n_k]$ be a union of chains of size $n$ whose
elements are labeled consecutively within chains. Then
\begin{equation*}
	\det(M-\lambda \mathbbm{1}) = \prod_{\substack{ S \subseteq [n]\\ \text{$S$ upper set in $P$}}}
	(\lambda - x_S)^{\mathfrak d_{P \setminus S}},
\end{equation*}
where $\mathfrak{d}_\emptyset =1$.
\end{thm}

Note that the antichain is a special case of a rooted forest and in
particular a union of chains.  In this case the Markov chain is the
Tsetlin library and all subsets of $[n]$ are upper (and lower)
sets. Hence Theorem~\ref{theorem.eigenvalues} specializes to the
results of Donnelly~\cite{donnelly.1991}, Kapoor and
Reingold~\cite{kapoor_reingold.1991}, and
Phatarford~\cite{phatarfod.1991} for the Tsetlin library.

The case of unions of chains, which are consecutively labeled, can be
interpreted as looking at a parabolic subgroup of $S_n$. If there are
$k$ chains of lengths $n_i$ for $1\le i \le k$, then the parabolic
subgroup is $S_{n_1} \times \cdots \times S_{n_k}$. In the realm of
the Tsetlin library, there are $n_i$ books of the same color. The
Markov chain consists of taking a book at random and placing it at the
end of the stack.

%%%%%%%%%%%%%%%%%%%%%%%%%%%%%%%%%%%%%%%%%%%%%%%%%%%
\section{$\R$-trivial monoids}
\label{section.R trivial}
%%%%%%%%%%%%%%%%%%%%%%%%%%%%%%%%%%%%%%%%%%%%%%%%%%%

In this section we briefly outline the proof of Theorem~\ref{theorem.eigenvalues}. More details can be found
in~\cite{AKS:2012}.

A finite \textbf{monoid} $\Monoid$ is a finite set with an associative multiplication and an identity element.
Green~\cite{green.1951} defined several preorders on $\Monoid$. 
In particular for $x,y\in \Monoid$ the $\R$- and $\LL$-order is defined as
\begin{equation}
\label{equation.LR order}
\begin{split}
	x \ge_{\R} y & \quad \text{if $y=xu$ for some $u\in \Monoid$,}\\
	x \ge_{\LL} y & \quad \text{if $y=ux$ for some $u\in \Monoid$.}
\end{split}	
\end{equation}
This ordering gives rise to equivalence classes ($\R$-classes or $\LL$-classes) 
\begin{equation*}
\begin{split}
	x \; \R \; y &\quad \text{if and only if $x\Monoid = y\Monoid$,}\\
	x \; \LL \; y &\quad \text{if and only if $\Monoid x = \Monoid y$.}
\end{split}
\end{equation*}
The monoid $\Monoid$ is said to be \textbf{$\R$-trivial} (resp. \textbf{$\LL$-trivial}) if all $\R$-classes (resp. $\LL$-classes)
have cardinality one.

Now let $P$ be a rooted forest of size $n$ and $\prom_i$ for $1\le
i\le n$ the operators on $\L(P)$ defined by the promotion graph of
Section~\ref{subsection.promotion}. That is, for $\pi,\pi' \in \L(P)$,
the operator $\prom_i$ maps $\pi$ to $\pi'$ if $\pi' = \pi
\partial_{\pi^{-1}_i}$.  We are interested in the monoid
$\Monoid^{\prom}$ generated by $\{\prom_i \mid 1\le i \le n\}$.
The next lemma shows that the action of the generators $\prom_i$
for rooted forests is very similar to the action of the operators of the Tsetlin
library by moving the letter $i$ to the end; the difference in this case is that
letters above $i$ need to be reordered according to the poset.

\begin{lem}
\label{lemma.action of del}
Let $P$ and $\prom_i$ be as above, and $\pi \in \L(P)$. Then $\pi \prom_i$ is the linear extension in $\L(P)$
obtained from $\pi$ by moving the letter $i$ to position $n$ and reordering all letters $j \succeq i$.
\end{lem}

\begin{eg}
Let $P$ be the union of a chain of length 3 and a chain of length 2, where the first chain is labeled by the 
elements $\{1,2,3\}$ and the second chain by $\{4,5\}$. Then $41235 \; \prom_1 = 41253$, which is obtained by moving
the letter 1 to the end of the word and then reordering the letters $\{1,2,3\}$, so that the result is again a linear extension
of $P$.
\end{eg}

Let $M$ be the transition matrix of the promotion graph of
Section~\ref{subsection.promotion}. Define $\Monoid$ to be the monoid
generated by $\{ G_i \mid 1\le i \le n\}$, where $G_i$ is the matrix
$M$ evaluated at $x_i =1$ and all other $x_j=0$.  We are now ready to
state the main result of this section.

\begin{thm} \label{theorem.r trivial}
$\Monoid$ is $\R$-trivial.
\end{thm}

\begin{rem} \label{remark.L trivial}
Considering the matrix monoid $\Monoid$ is equivalent to considering the abstract monoid $\Monoid^{\prom}$ 
generated by $\{\prom_i \mid 1\le i \le n\}$. Since the operators $\prom_i$ act on the right on linear extensions, the 
monoid $\Monoid^{\prom}$ is $\LL$-trivial instead of $\R$-trivial.
\end{rem}

The proof of Theorem~\ref{theorem.r trivial} exploits Lemma~\ref{lemma.action of del} by proving that there is 
an order on idempotents using right factors. For $x\in \Monoid^{\prom}$,
let $\rfactor(x)$ be the maximal common right factor of all elements in the image of $x$, that
is, all elements $\pi \in \im(x)$ can be written as $\pi = \pi_1 \cdots  \pi_m \rfactor(x)$ and there is no
bigger right factor for which this is true. Let us also define the set of entries in the right
factor $\Rfactor(x) = \{ i \mid i\in\rfactor(x) \}$. Note that since all elements in the image set of $x$
are linear extensions of $P$, $\Rfactor(x)$ is an upper set of $P$.
Theorem~\ref{theorem.r trivial} is then established by showing that for idempotents $x$, the set
$\Rfactor(x)$ is the same as the left stabilizer $\{i \mid \prom_i x = x\}$ which imposes a partial order.

\begin{eg} \label{example.r trivial monoid}
Let $P$ be the poset on three elements $\{1,2,3\}$, where $2$ covers $1$ and there are no further relations.
The linear extensions of $P$ are $\{123,132,312\}$. 
The monoid $\Monoid$ with $\R$-order, where an edge labeled $i$ means right multiplication by $G_i$,
is depicted in Figure~\ref{figure.monoid}. From the picture it is clear that the elements in the monoid are partially ordered.
\end{eg}
This confirms Theorem~\ref{theorem.r trivial} that the monoid is $\R$-trivial.
\begin{figure}[h]
\begin{center}
\begin{tikzpicture}[>=latex,line join=bevel,scale=0.7]
\node (000111000) at (34bp,322bp) [draw,draw=none] {$\begin{array}{l}\verb|[0|\phantom{x}\verb|0|\phantom{x}\verb|0]|\\\verb|[1|\phantom{x}\verb|1|\phantom{x}\verb|1]|\\\verb|[0|\phantom{x}\verb|0|\phantom{x}\verb|0]|\end{array}$};
  \node (000100011) at (156bp,222bp) [draw,draw=none] {$\begin{array}{l}\verb|[0|\phantom{x}\verb|0|\phantom{x}\verb|0]|\\\verb|[1|\phantom{x}\verb|0|\phantom{x}\verb|0]|\\\verb|[0|\phantom{x}\verb|1|\phantom{x}\verb|1]|\end{array}$};
  \node (100010001) at (156bp,22bp) [draw,draw=none] {$\begin{array}{l}\verb|[1|\phantom{x}\verb|0|\phantom{x}\verb|0]|\\\verb|[0|\phantom{x}\verb|1|\phantom{x}\verb|0]|\\\verb|[0|\phantom{x}\verb|0|\phantom{x}\verb|1]|\end{array}$};
  \node (111000000) at (226bp,122bp) [draw,draw=none] {$\begin{array}{l}\verb|[1|\phantom{x}\verb|1|\phantom{x}\verb|1]|\\\verb|[0|\phantom{x}\verb|0|\phantom{x}\verb|0]|\\\verb|[0|\phantom{x}\verb|0|\phantom{x}\verb|0]|\end{array}$};
  \node (000110001) at (50bp,122bp) [draw,draw=none] {$\begin{array}{l}\verb|[0|\phantom{x}\verb|0|\phantom{x}\verb|0]|\\\verb|[1|\phantom{x}\verb|1|\phantom{x}\verb|0]|\\\verb|[0|\phantom{x}\verb|0|\phantom{x}\verb|1]|\end{array}$};
  \node (000000111) at (228bp,322bp) [draw,draw=none] {$\begin{array}{l}\verb|[0|\phantom{x}\verb|0|\phantom{x}\verb|0]|\\\verb|[0|\phantom{x}\verb|0|\phantom{x}\verb|0]|\\\verb|[1|\phantom{x}\verb|1|\phantom{x}\verb|1]|\end{array}$};
  \draw [green,<-] (111000000) ..controls (274.92bp,124.11bp) and (278bp,123.24bp)  .. (278bp,122bp) .. controls (278bp,120.01bp) and (270.13bp,118.97bp)  .. (111000000);
  \definecolor{strokecol}{rgb}{0.0,0.0,0.0};
  \pgfsetstrokecolor{strokecol}
  \draw (287bp,122bp) node {$3$};
  \draw [red,<-] (111000000) ..controls (284.71bp,139.19bp) and (296bp,133.4bp)  .. (296bp,122bp) .. controls (296bp,107.97bp) and (278.89bp,102.44bp)  .. (111000000);
  \draw (305bp,122bp) node {$2$};
  \draw [blue,<-] (111000000) ..controls (293.25bp,144.84bp) and (314bp,138.26bp)  .. (314bp,122bp) .. controls (314bp,103.42bp) and (286.9bp,97.471bp)  .. (111000000);
  \draw (323bp,122bp) node {$1$};
  \draw [red,<-] (000110001) ..controls (98.925bp,128.34bp) and (102bp,125.72bp)  .. (102bp,122bp) .. controls (102bp,116.04bp) and (94.127bp,112.92bp)  .. (000110001);
  \draw (111bp,122bp) node {$2$};
  \draw [green,<-] (000111000) ..controls (82.925bp,324.11bp) and (86bp,323.24bp)  .. (86bp,322bp) .. controls (86bp,320.01bp) and (78.127bp,318.97bp)  .. (000111000);
  \draw (95bp,322bp) node {$3$};
  \draw [red,<-] (000111000) ..controls (92.706bp,339.19bp) and (104bp,333.4bp)  .. (104bp,322bp) .. controls (104bp,307.97bp) and (86.892bp,302.44bp)  .. (000111000);
  \draw (113bp,322bp) node {$2$};
  \draw [blue,<-] (000111000) ..controls (101.25bp,344.84bp) and (122bp,338.26bp)  .. (122bp,322bp) .. controls (122bp,303.42bp) and (94.897bp,297.47bp)  .. (000111000);
  \draw (131bp,322bp) node {$1$};
  \draw [red,<-] (000000111) ..controls (190.08bp,289.62bp) and (186.84bp,285.85bp)  .. (184bp,282bp) .. controls (175.41bp,270.34bp) and (168.49bp,255.62bp)  .. (000100011);
  \draw (193bp,272bp) node {$2$};
  \draw [blue,<-] (000000111) ..controls (213.54bp,281.1bp) and (208.34bp,270.66bp)  .. (202bp,262bp) .. controls (197.25bp,255.51bp) and (191.26bp,249.35bp)  .. (000100011);
  \draw (221bp,272bp) node {$1$};
  \draw [red,<-] (000110001) ..controls (97.168bp,77.502bp) and (117.64bp,58.193bp)  .. (100010001);
  \draw (121bp,72bp) node {$2$};
  \draw [blue,<-] (000100011) ..controls (156bp,148.97bp) and (156bp,79.056bp)  .. (100010001);
  \draw (165bp,122bp) node {$1$};
  \draw [green,<-] (000111000) ..controls (39.818bp,249.27bp) and (45.43bp,179.13bp)  .. (000110001);
  \draw (53bp,222bp) node {$3$};
  \draw [green,<-] (111000000) ..controls (194.12bp,76.461bp) and (181.06bp,57.8bp)  .. (100010001);
  \draw (206bp,72bp) node {$3$};
  \draw [green,<-] (000000111) ..controls (276.92bp,324.11bp) and (280bp,323.24bp)  .. (280bp,322bp) .. controls (280bp,320.01bp) and (272.13bp,318.97bp)  .. (000000111);
  \draw (289bp,322bp) node {$3$};
  \draw [red,<-] (000000111) ..controls (286.71bp,339.19bp) and (298bp,333.4bp)  .. (298bp,322bp) .. controls (298bp,307.97bp) and (280.89bp,302.44bp)  .. (000000111);
  \draw (307bp,322bp) node {$2$};
  \draw [blue,<-] (000000111) ..controls (295.25bp,344.84bp) and (316bp,338.26bp)  .. (316bp,322bp) .. controls (316bp,303.42bp) and (288.9bp,297.47bp)  .. (000000111);
  \draw (325bp,322bp) node {$1$};
  \draw [green,<-] (000111000) ..controls (87.65bp,278.02bp) and (111.61bp,258.39bp)  .. (000100011);
  \draw (114bp,272bp) node {$3$};
  \draw [blue,<-] (000100011) ..controls (108.83bp,177.5bp) and (88.364bp,158.19bp)  .. (000110001);
  \draw (121bp,172bp) node {$1$};
\end{tikzpicture}
\end{center}
\caption{Monoid $\Monoid$ in right order for the poset of Example~\ref{example.r trivial monoid}.
With the conventions in~\eqref{equation.LR order}, the identity is the biggest element in $\R$-order.
\label{figure.monoid}}
\end{figure}
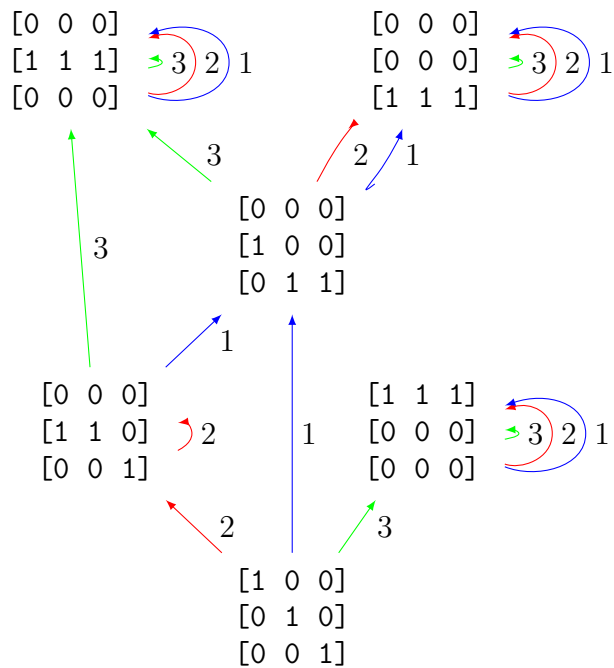
The proof of Theorem~\ref{theorem.eigenvalues} now follows from~\cite[Theorems 6.3 and 6.4]{steinberg.2006}
and some further considerations regarding the lattice $L$. For more details see~\cite[Section 6]{AKS:2012}.

%%%%%%%%%%%%%%%%%%%%%%%%%%%%%%%%%%%%%%%%%%%%%%%%%%%%%%%%%%
\section{Mixing Times} \label{section.mixing times}
%%%%%%%%%%%%%%%%%%%%%%%%%%%%%%%%%%%%%%%%%%%%%%%%%%%%%%%%%%

For random walks on hyperplane arrangements, Brown and Diaconis~\cite{brown_diaconis.1998}
(see also~\cite{athanasiadis_diaconis.2010}) give explicit bounds for the rates of convergence  to
stationarity. These bounds still hold for Markov chains related to left-regular bands~\cite{brown.2000}.
Here we present analogous results for the Markov chains corresponding
to the $\R$-trivial monoids of Section~\ref{section.chains}. The methods are very similar
to the ones we used for Markov chains related to nonabelian sandpile models~\cite{ASST:2013}, which also
turn out to yield $\R$-trivial monoids.

The {\bf rate of convergence} is the total variation distance from stationarity after $k$ steps, that is,
\[
	|| \mathbb{P}^k-w ||=\frac{1}{2}\sum_{\pi \in \L(P)}| \mathbb{P}^k(t)-w(\pi)|\;,
\] 
where $\mathbb{P}^k$ is the distribution after $k$ steps and $w$ is the stationary distribution.

\begin{thm}
  \label{theorem.rate of convergence}
  Let $P$ be a rooted forest with $n:=|P|$ and $p_x:=\min\{x_i \mid 1\le i \le n\}$. Then, as
  soon as $k\ge (n^2-1)/p_x$, the distance to stationarity of the promotion Markov chain satisfies
  \begin{displaymath}
    || \mathbb{P}^k-w || \leq \exp\left(- \frac{(kp_x-(n^2-1))^2}{2kp_x}\right)\,.
  \end{displaymath}
\end{thm}

The {\bf mixing time}~\cite{levin_peres_wilmer.2009} is the number of steps $k$ until $|| \mathbb{P}^k-w || \le e^{-c}$
(where different authors use different conventions for the value of $c$). Using
Theorem~\ref{theorem.rate of convergence} we require
\[
(kp_x-(n^2-1))^2 \ge 2kp_xc\,,
\]
which shows that the mixing time is at
most $\frac{2(n^2+c-1)}{p_x}$.
If the probability distribution $\{x_i \mid 1\le i \le n\}$ is uniform, then $p_x$ is of order $1/n$ and the mixing time is
of order at most $n^3$.

The proof of Theorem~\ref{theorem.rate of convergence} follows the same outline as the proof in~\cite[Section 5.3]{ASST:2013}.
We need to define a statistic $u(x)$ for $x \in \Monoid$ such that $u(x)$ is minimal if and only if $x$ is the constant map
and furthermore
\begin{enumerate}
\item \label{i.decrease}
\textbf{$u$ decreases along $\R$-order:} $u(xx')\leq u(x)$ for any $x,x'\in \Monoid$. 
\item \label{i.strict decrease}
\textbf{Existence of generator with strict decrease:} There exists a generator $G_i$ such that 
  $u(xG_i)<u(x)$.
\end{enumerate}
Unlike in~\cite{ASST:2013}, we take $u(x) \in \mathbb{Z}_{\ge 0}^2$ with lexicographic ordering on $\mathbb{Z}_{\ge 0}^2$,
that is $(x,y)<(x',y')$ if either $x<x'$, or $x=x'$ and $y<y'$.
Set $u(x):= (n-|\Rfactor(x)|, |\des(x)|)$, where $\des(x) = \{ i \mid x G_i = x\}$. It is clear that $u(x)=(0,n)$ if and only if $x$ is 
a constant map, which is the minimal value $u$ can achieve. The maximal value of $u$ is achieved by the identity
$u(e)=(n,0)$.
The two conditions follow from~\cite[Section 6]{AKS:2012}: either the right factor $\rfactor(x)$ increases by right multiplication
by a generator $G_i$; if not, then $\{i\} \cup \Rfactor(x)$ must be an upper set again and $\des(xG_i) = \des(x) \setminus \{j \mid
\text{$j$ covers $i$ in $P$}\}$.

Therefore, the probability that $(n,0) \ge u(x)>(0,n)$ after $k$ steps of the right
random walk on $\Monoid$ is bounded above by the probability of having at most $(n+1)(n-1)=n^2-1$ successes in $k$ 
Bernoulli trials with success probability $p_x$. A successful step is one that decreases the statistic $u$.
Using Chernoff's inequality for the cumulative distribution function of a binomial random variable as in~\cite{ASST:2013}
we obtain Theorem~\ref{theorem.rate of convergence}.

%%%%%%%%%%%%%%%%%%%%%%%%%%%%%%%%%%%%%%%%%%%%%%%%%%%%%%%%%%
\section{Other Posets} \label{section.other posets}
%%%%%%%%%%%%%%%%%%%%%%%%%%%%%%%%%%%%%%%%%%%%%%%%%%%%%%%%%%

So far \cite{AKS:2012}, we have characterized posets, where the
Markov chains for the promotion graph yield certain simple formulas
for their eigenvalues and multiplicities. The eigenvalues have
explicit expressions for rooted forests and there is an explicit
combinatorial interpretation for the multiplicities as derangement
numbers of permutations for unions of chains by
Theorem~\ref{theorem.poset derangements}.

However, we have not classified all possible posets, whose
promotion graphs have nice properties. For example, the non-zero
eigenvalues of the transition matrix of the promotion graph of the
poset in Example~\ref{example.running example} are given by
\[
x_3+x_4, \quad x_3, \quad 0 \quad \text{and} \quad -x_1\;,
\]
even though the corresponding monoid is not $\R$-trivial (in fact, it
is not even aperiodic). The egg-box picture of the monoid is given in
Figure~\ref{figure.eggbox}. Notice that one of the eigenvalues is
negative.

\begin{figure}
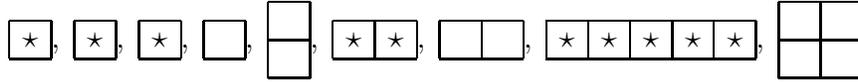

\[
\begin{array}{|c|}\hline \star \\ \hline \end{array},\;
\begin{array}{|c|}\hline \star \\ \hline \end{array},\;
\begin{array}{|c|}\hline \star \\ \hline \end{array},\;
\begin{array}{|c|}\hline  \phantom{\star} \\ \hline \end{array},\;
\begin{array}{|c|}\hline \phantom{\star} \\ \hline \phantom{\star} \\ \hline \end{array},\;
\begin{array}{|c|c|}\hline \star & \star \\ \hline \end{array},\;
\begin{array}{|c|c|}\hline \phantom{\star} & \phantom{\star} \\ \hline \end{array},\;
\begin{array}{|c|c|c|c|c|}\hline \star & \star & \star & \star & \star \\ \hline \end{array},\;
\begin{array}{|c|c|} \hline \phantom{\star} & \phantom{\star} \\ \hline \phantom{\star} & \phantom{\star} \\ \hline \end{array}
\]
\caption{Egg-box picture for the monoid associated to the promotion Markov chain for the poset in
Example~\ref{example.running example}.
\label{figure.eggbox}
}
\end{figure}

\begin{table}[h!]
\begin{tabular}{|c|c|}
\hline
Poset & Eigenvalues (other than 1) \\
\hline
{\tiny \begin{tikzpicture}
\draw (-1,1) -- (0,0) -- (1,1);
\draw (-1,1.25) node {$2$};
\draw (0,-.25) node {$1$};
\draw (1,1.25) node {$3$};
\draw (2,1.25) node {$4$};
\end{tikzpicture}}
& 
$0,\; 0,\; 0,\; x_2,\; x_3,\; x_2+x_3,\; x_4-x_1$ \\ \hline
{\tiny \begin{tikzpicture}
\draw (0,0) -- (0,.75);
\draw (-1,2) -- (0,1.25) -- (1,2);
\draw (0,-.25) node {$1$};
\draw (0,1) node {$2$};
\draw (-1,2.25) node {$3$};
\draw (1,2.25) node {$4$};
\end{tikzpicture}}
& $-x_1-x_2$ \\ \hline
{\tiny \begin{tikzpicture}
\draw (0,1) -- (0,0) -- (1,1) -- (1,0);
\draw (0,1.25) node {$4$};
\draw (0,-.25) node {$1$};
\draw (1,1.25) node {$3$};
\draw (1,-.25) node {$2$};
\end{tikzpicture}}
& $0,\;  x_3,\;  -x_1,\;  x_3+x_4$ \\ \hline
{\tiny \begin{tikzpicture}
\draw (-1,1.25) -- (0,2) -- (1,1.25);
\draw (-1,.75) -- (0,0) -- (1,.75);
\draw (0,-.25) node {$1$};
\draw (-1,1) node {$2$};
\draw (1,1) node {$3$};
\draw (0,2.25) node {$4$};
\end{tikzpicture}}
& $x_4-x_1$ \\ \hline
{\tiny \begin{tikzpicture}
\draw (0,0) -- (0,1) -- (1,0) -- (1,1) -- (0,0);
\draw (0,-.25) node {$1$};
\draw (1,-.25) node {$2$};
\draw (0,1.25) node {$3$};
\draw (1,1.25) node {$4$};
\end{tikzpicture}}
& $0,\;  x_3+x_4,\; -x_1-x_2$ \\
\hline
\end{tabular}
\caption{All inequivalent posets of size 4 whose promotion transition
  matrices have simple expressions for their eigenvalues.
\label{table.notdf4 nice}}
\end{table}

On the other hand, not all posets have this property. In particular, the
poset with covering relations $1<2,1<3$ and $1<4$ has six linear
extensions, but the characteristic polynomial of its transition matrix
does not factorize at all. It would be interesting to classify all
posets with the property that all the eigenvalues of the transition
matrices of the promotion Markov chain are linear in the probability
distribution $x_i$. In such cases, one would also like an explicit formula for
the multiplicity of these eigenvalues. 

We list all posets of size 4, which are not down forests and which
nonetheless have simple linear expressions for their eigenvalues in
Table~\ref{table.notdf4 nice} along with the eigenvalues. For all such
posets, there is at least one eigenvalue which contains a negative
term.  The posets, which are not down forests and the eigenvalues of
whose promotion transition matrices have nonlinear expressions, are
given in Table~\ref{table.notdf4 not nice}. Comparing the two tables,
it is not obvious how to characterize those posets where the
eigenvalues are simple. It would be interesting to classify posets
where all eigenvalues are linear in the parameters and understand the
eigenvalues and their multiplicities completely. 
For comparison, the egg-box picture of the second poset in Table~\ref{table.notdf4 not nice}
is presented in Figure~\ref{figure.eggbox bad}.

\begin{table}[h!]
\begin{tabular}{|c|c|}
\hline
{\tiny \begin{tikzpicture}
\draw (-1,1) -- (0,0) -- (1,1);
\draw (0,0) -- (0,1);
\draw (0,-.25) node {$1$};
\draw (-1,1.25) node {$2$};
\draw (0,1.25) node {$3$};
\draw (1,1.25) node {$4$};
\end{tikzpicture}}
&
{\tiny \begin{tikzpicture}
\draw (-1,.75) -- (0,0) -- (1,.75);
\draw (-1,1.25) -- (-1,2);
\draw (0,-.25) node {$1$};
\draw (-1,1) node {$2$};
\draw (1,1) node {$3$};
\draw (-1,2.25) node {$4$};
\end{tikzpicture}}
\\ \hline
\end{tabular}
\caption{All inequivalent posets of size 4 whose promotion transition
  matrices do not have simple expressions for their eigenvalues.
\label{table.notdf4 not nice}}
\end{table}

\begin{figure}
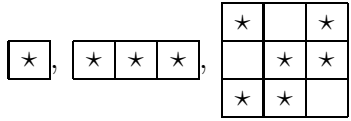

\[
\begin{array}{|c|}\hline \star \\ \hline \end{array},\;
\begin{array}{|c|c|c|}\hline \star & \star & \star \\ \hline \end{array},\;
\begin{array}{|c|c|c|} \hline \star & \phantom{\star} & \star \\ \hline \phantom{\star} & \star & \star \\ \hline  \star & \star & \phantom{\star} \\ \hline \end{array}
\]
\caption{Egg-box picture for the monoid associated to the promotion Markov chain for the second poset in
Table~\ref{table.notdf4 not nice}.
\label{figure.eggbox bad}
}
\end{figure}

Using data from all posets which are not down forests of sizes up to
7, we have the following necessary (but not sufficient) conjecture.

\begin{conj}\label{conjecture.not df}
Let $P$ be a poset of size $n$ which is not a down forest and $M$ be
its promotion transition matrix. If $M$ has eigenvalues which are
linear in the parameters $x_1, \dots, x_n$, then the following hold
\begin{enumerate}
\item the coefficients of the parameters in the eigenvalues are only
  one of $\pm 1$,
\item each element of $P$ has at most two successors,
\item the only parameters whose coefficients in the eigenvalues are -1
  are those which either have two successors or one of whose
  successors have two successors.
\end{enumerate}
\end{conj}

%%%%%%%%%%%%%%%%%%%%%%%%%%%%%%%%%%%%%%%%%%%%%%%%%%%
\section{Subsets of $S_n$} \label{section.subsets}
%%%%%%%%%%%%%%%%%%%%%%%%%%%%%%%%%%%%%%%%%%%%%%%%%%%
We define a generalization of the action of promotion on an arbitrary
nonempty subset of $S_n$ inspired by the ideas in \cite{haiman.1992,
  malvenuto_reutenauer.1994,stanprom}.  Let $A$ be such a subset and
suppose $\pi=\pi_1 \cdots \pi_n \in A$ in one-line notation. Then
we define the operator $\sigma_i$ for $i \in \{1,\dots,n\}$ as
\begin{equation} \label{defsig}
	\pi \sigma_i = \begin{cases}
	\pi_1 \cdots \pi_{i-1} \pi_{i+1} \pi_i \cdots \pi_n & 
        \text{if $\pi_1 \cdots \pi_{i-1} \pi_{i+1} \pi_i \cdots \pi_n \in A$ }\\
	\pi & \text{otherwise.} 
\end{cases}
\end{equation}
In other words, $\sigma_i$ acts non-trivially on a permutation in $A$
if interchanging entries $\pi_i$ and $\pi_{i+1}$ yields another
permutation in $A$, and otherwise acts as the identity.  Then the
generalized promotion operator, also denoted $\partial_j$, is an
operator on $A$ defined by
\begin{equation}
\label{equation.promotion sigma}
 	\partial_j = \sigma_j \sigma_{j+1} \cdots \sigma_{n-1}.
\end{equation}

As in Sections~\ref{subsection.promotion uniform} and~\ref{subsection.promotion}, 
we can define a promotion graph whose
vertices are the elements of the set $A$ and where there is an edge
between permutations $\pi$ and $\pi'$ if and only if $\pi' = \pi
\partial_j$. In the uniform promotion case, such an edge has weight
$x_j$ and in the promotion case, the edge has weight $x_{\pi_j}$. In
both cases, we have analogous Markov chains. We describe the
stationary distribution of these chains below.

\begin{thm} \label{theorem.gen promotion}
Assuming the promotion graph for $A$ is strongly connected, the unique
stationary state weight $w(\pi)$ of the permutation $\pi\in A$ for the
corresponding discrete time Markov chain is
\begin{enumerate}
\item in the uniform promotion case
\be \label{subset.uniform}
  w(\pi) = \frac 1{|A|},
\ee
\item in the promotion case
  \be \label{formulaIII}
  w(\pi) = \prod_{i=1}^{n} \frac{x_{1}+ \cdots + x_{i}}
  {x_{\pi_1}+ \cdots + x_{\pi_i}}\;.
  \ee
\end{enumerate}
\end{thm}

The proofs are essentially identical to the proofs of
Theorem~\ref{theorem.uniform promotion} and
Theorem~\ref{theorem.promotion} given in \cite{AKS:2012} and are
skipped.

\begin{rem} \label{remark.stationary subsets}
\mbox{}
\begin{enumerate}
  \item The entries of $w$ do not, in general, sum to one. Therefore
    this is not a true probability distribution, but this is easily
    remedied by a multiplicative constant $Z_{A}$ depending only on
    the subset $A$.

  \item Even if the set $A$ is such that the promotion graph is not
    strongly connected, \eqref{subset.uniform} and \eqref{formulaIII}
    hold. However, the formula need not be unique. The proofs of
    Theorem~\ref{theorem.gen promotion} still go through because all
    we need to do is to verify that the master equation
    \eqref{master.equation} holds.
\end{enumerate}
\end{rem}

There is a natural way to build subsets $A$ which cannot be the set of
linear extensions $\L(P)$ for any poset $P$, and whose promotion
graphs are yet strongly connected. The idea is to consider a union of
sorting networks.  A {\bf sorting network} from the identity
permutation $e$ to any permutation $\pi$ is a shortest path from one to
the other by a series of nearest-neighbor transpositions. In other words,
these are maximal chains in right weak order starting at the identity. For example,
one sorting network to the permutation $24153$ is
\[
1 2 3 4 5 \to 1 2 4 3 5 \to 2 1 4 3 5 \to 2 4 1 3 5 \to 2 4 1 5 3.
\]

\begin{prop}
\label{proposition.subset strongly_connected}
Let $A \subset S_n$ be a union of sorting networks.
Consider the digraph $G_A$ whose vertices are labeled by the elements of
$A$ and whose edges are given as follows: for $\pi, \pi'
\in A$, there is an edge between $\pi$ and $\pi'$ in $G_A$ if and only
if $\pi' = \pi \partial_{j}$ for some $j \in [n]$. Then $G_A$ is
strongly connected.
\end{prop}

\begin{proof}
The operators $\partial_i$ are each invertible, which means that each
vertex of $G_A$ has exactly one edge pointing in and one pointing out
for each $i$. Therefore, it suffices to show that there is a directed
path from $e$ to $\pi$ for every $\pi$ in $A$.

By definition of a sorting network, $\pi$ can be written as
$e \sigma_{i_k} \dots \sigma_{i_1}$. Although the action of each
$\sigma_{i_j}$ depends crucially on the set $A$, they satisfy
$\sigma_{i_j}^2=1$. Using the fact that $\sigma_{n-1} =
\partial_{n-1}$ and \eqref{equation.promotion sigma}, one can
recursively express each $\sigma_{i_j}$ as a product of
$\partial_\ell$'s analogous to the proof of Lemma~2.3 in
\cite{AKS:2012}.
\end{proof}

As a consequence of Proposition~\ref{proposition.subset strongly_connected},
the unique stationary distribution of a subset
which is a union of sorting networks is given by
\eqref{formulaIII}. One is naturally led to ask whether the
eigenvalues of these transition matrices are also linear in the
parameters. This does not seem to be true in any general sense.

%%%%%%%%%%%%%%%%%%%%%%%%%%%%%%%%%%%%%%%%%%%%%%%%%%%%%%%%%%
\bibliographystyle{alpha} \newcommand{\etalchar}[1]{$^{#1}$}

%%%%%%%%%%%%%%%%%%%%%%%%%%%%%%%%%%%%%%%%%%%%%%%%%%%%%%%%%%

\end{document}